\documentclass[10pt,reqno]{amsart}

\usepackage{amsmath}
\usepackage{amsthm}
\usepackage{amssymb}
\usepackage{amsfonts}
\usepackage{amsxtra}
\usepackage{fullpage}
\usepackage{amscd}
\let\nc\newcommand

\theoremstyle{plain}
\newtheorem*{thm}{Theorem}
\newtheorem*{prop}{Proposition}
\newtheorem*{cor}{Corollary}
\newtheorem*{lem}{Lemma}

\theoremstyle{definition}
\newtheorem*{defn}{Definition}
\newtheorem*{example}{Example}
\newtheorem*{remark}{Remark}
\newtheorem*{Main}{Theorem}
\newtheorem*{MainCor}{Corollary}
\nc{\bdm}{\begin{displaymath}}
\nc{\edm}{\end{displaymath}}
\nc{\bthm}{\begin{thm}}
\nc{\ethm}{\end{thm}}
\nc{\blem}{\begin{lem}}
\nc{\elem}{\end{lem}}
\nc{\bcor}{\begin{cor}}
\nc{\ecor}{\end{cor}}
\nc{\beq}{\begin{equation}}
\nc{\eeq}{\end{equation}}
\nc{\bprop}{\begin{prop}}
\nc{\eprop}{\end{prop}}
\nc{\bdefn}{\begin{defn}}
\nc{\edefn}{\end{defn}}

\nc{\Z}{\mathbb{Z}}
\newcommand{\N}{\mathbb{N}}

\newcommand{\C}{\mathbb{C}}
\newcommand{\h}{\mathfrak{h}}

\renewcommand{\Im}{\mbox{\textrm{Im}}\,}
\renewcommand{\ker}{\textrm{Ker}}
\newcommand{\Ker}{\mbox{\textrm{Ker}}\,}

\newcommand{\Fun}{\mbox{\textrm{Fun}}\,}
\newcommand{\Kldim}{\mbox{\textrm{Kl.dim}}\,}
\newcommand{\height}{\mbox{\textrm{ht}}\,}
\nc{\Hom}{\textrm{Hom}}
\nc{\rank}{\textrm{rank} \,}
\nc{\ds}{\dots}

\let\mc\mathcal
\let\mf\mathfrak

\nc{\HW}{\bar{H}_{\mathbf{c}}(W)}
\nc{\HK}{\bar{H}_{\mathbf{c}}(K)}
\nc{\HtK}{\widetilde{H}_{\mathbf{c}}(K)}
\nc{\CMW}{\textsf{CM}_{\mbf{c}}(W)}
\nc{\CMK}{\textsf{CM}_{\mbf{c}}(K)}
\nc{\mbf}{\mathbf}
\nc{\LK}{\textsf{Irr}(K)}
\nc{\LW}{\textsf{Irr}(W)}
\nc{\Res}{\mathsf{Res} \, }
\nc{\Ind}{\mathsf{Ind} \, }

\nc{\cont}{\textrm{cont}}
\renewcommand{\mod}{\textrm{mod}}
\nc{\eWb}{\mathbf{e}_{W_b}}
\nc{\eW}{\mathbf{e}_{W}}
\nc{\msf}{\mathsf}
\nc{\Ui}{\mc{U}_{i,+}}
\nc{\Uone}{\mc{U}_{1,+}}
\nc{\Utwo}{\mc{U}_{2,+}}
\newcommand{\mmod}{\text{-}\mathrm{mod}}
\nc{\minusone}{-1}
\nc{\minustwo}{-2}
\nc{\p}{\partial}

\begin{document}

\title{Cuspidal representations of rational Cherednik algebras at $t = 0$}

\author{Gwyn Bellamy}

\address{School of Mathematics and Maxwell Institute for Mathematical Sciences, University of Edinburgh, James Clerk Maxwell Building, Kings Buildings, Mayfield Road, Edinburgh EH9 3JZ, Scotland}
\email{G.E.Bellamy@sms.ed.ac.uk}

\begin{abstract}
\noindent We study those finite dimensional quotients of the rational Cherednik algebra at $t = 0$ that are supported at a point of the centre. It is shown that each such quotient is Morita equivalent to a certain ``cuspidal'' quotient of a rational Cherednik algebra associated to a parabolic subgroup of $W$.
\end{abstract}

\maketitle

\section{Introduction}

\subsection{} Let $W$ be a finite complex reflection group. Associated to $W$ is a family of noncommutative algebras, the rational Cherednik algebras. These algebras $H_{t,\mbf{c}}(W)$ depend on a pair of parameters, $t$ and $\mbf{c}$ (precise definitions are given in (\ref{subsection:defns})). At $t = 0$ the algebras are finite modules over their centres. The aim of this paper is to continue the study of finite dimensional quotients of the rational Cherednik algebra at $t = 0$. Using certain completions of the centre of the rational Cherednik algebra we are able to relate the symplectic leaves of the corresponding generalized Calogero-Moser space $X_{\mbf{c}}(W)$ to zero dimensional leaves in the generalized Calogero-Moser space of a parabolic subgroup of $W$. As a consequence of this we are able to relate the finite dimensional quotients supported on a point of a given leaf to finite dimensional algebras supported on a zero dimensional leaf associated to the parabolic subgroup of $W$. To be precise, let $\mathcal{L}$ be a symplectic leaf in $X_{\mathbf{c}}(W)$ of dimension $2l$ and $\chi$ a point on $\mathcal{L}$. If $\mf{m}_{\chi}$ is the maximal ideal of the centre of the rational Cherednik algebra defining $\chi$ then set $H_{\mbf{c},\chi} := H_{0,\mbf{c}}(W) / \mf{m}_{\chi} \cdot H_{0,\mbf{c}}(W)$, a finite dimensional algebra. Our main results says:

\begin{Main}
There exists a parabolic subgroup $W_b$, $b \in \h$, of $W$ of rank $\dim \h - l$ and cuspidal algebra $H_{\mathbf{c}',\psi}$ with $\psi \in X_{\mathbf{c}'}(W_b)$ such that
\begin{displaymath}
H_{\mathbf{c},\chi} \simeq \textrm{Mat}_{\, |W / W_b|}\,(H_{\mathbf{c}',\psi}).
\end{displaymath}
\end{Main}

\noindent Here cuspidal means that the point $\psi$ defines a zero dimensional leaf $\{ \psi \}$ in $X_{\mathbf{c}'}(W_b)$. A consequence of this result is that

\begin{MainCor}
There exists a functor 
\bdm
\Phi_{\psi, \chi} \, : \, H_{\mathbf{c}',\psi} \mmod \stackrel{\sim}{\longrightarrow} H_{\mbf{c},\chi} \mmod
\edm
defining an equivalence of categories such that
\begin{displaymath}
\Phi_{\psi,\chi}(M) \simeq \textsf{Ind}_{\, W_b}^{\, W} \,  M \quad \forall \, M \in H_{\mathbf{c}',\psi} \mmod
\end{displaymath}
as $W$-modules.
\end{MainCor}

Since there are only finitely many zero dimensional leaves in $X_{\mbf{c}}(W)$ the above result shows that to describe the $W$-module structure of all the simple modules for a particular rational Cherednik algebra one only needs to describe the $W_b$-module structure of the cuspidal simple modules for each parabolic subgroup $W_b$ of $W$.

\section{The rational Cherednik algebra at $t = 0$}

\subsection{Definitions and notation}\label{subsection:defns}

Let $W$ be a complex reflection group, $\mathfrak{h}$ its reflection representation over $\C$ with rank $\mathfrak{h} = n$, and $\mathcal{S}(W)$ the set of all complex reflections in $W$. Let $( \cdot, \cdot ) : \mathfrak{h} \times \mathfrak{h}^* \rightarrow \C$ be the natural pairing defined by $(y,x) = x(y)$. For $s \in \mathcal{S}(W)$, fix $\alpha_s \in \mathfrak{h}^*$ to be a basis of the one dimensional space $\Im (s - 1)|_{\mathfrak{h}^*}$ and $\alpha_s^{\vee} \in \mathfrak{h}$ a basis of the one dimensional space $\Im (s - 1)|_{\mathfrak{h}}$, normalised so that $\alpha_s(\alpha_s^\vee) = 2$. Choose $\mathbf{c} : \mathcal{S}(W) \rightarrow \C$ to be a $W$-equivariant function and $t$ a complex number. The \textit{rational Cherednik algebra}, $H_{t,\mathbf{c}}(W)$, as introduced by Etingof and Ginzburg \cite[page 250]{1}, is the quotient of the skew group algebra of the tensor algebra, $T(\mf{h} \oplus \mf{h}^*) \rtimes W$, by the ideal generated by the relations

\begin{equation}\label{eq:rel}
[x_1,x_2] = 0, \qquad [y_1,y_2] = 0, \qquad [x_1,y_1] = t (y_1,x_1) - \sum_{s \in \mathcal{S}} \mathbf{c}(s) (y_1,\alpha_s)(\alpha_s^\vee,x_1) s, 
\end{equation}
\noindent for all $x_1,x_2 \in \mathfrak{h}^* \textrm{ and } y_1,y_2 \in \mathfrak{h}$.\\

\noindent For any $\nu \in \C \backslash \{ 0 \}$, the algebras $H_{\nu t,\nu \mathbf{c}}(W)$ and $H_{t,\mathbf{c}}(W)$ are isomorphic. In this article we will only consider the case $t = 0$, therefore we are free to rescale $\mbf{c}$ by $\nu$ whenever this is convenient.\\

\noindent A fundamental result for rational Cherednik algebras, proved by Etingof and Ginzburg \cite[Theorem 1.3]{1}, is that the PBW property holds for all $t, \mbf{c}$. That is, there is a vector space isomorphism 
\beq\label{eq:PBW}
H_{t, \mbf{c}}(W) \stackrel{\sim}{\rightarrow} \C [\h] \otimes \C W \otimes \C [\h^*].
\eeq

\subsection{The generalized Calogero-Moser Space}\label{sub:calogerospace}
The centre $Z_{\mathbf{c}}(W)$ of $H_{0,\mathbf{c}}(W,\mathfrak{h})$ is an affine domain. We shall denote by $X_{\mathbf{c}}(W) := \textrm{Spec}\, (Z_{\mathbf{c}}(W))$ the corresponding affine variety. The space $X_{\mathbf{c}}(W,\mathfrak{h})$ is called the \textit{generalized Calogero-Moser space} associated to the complex reflection group $W$ at parameter $\mathbf{c}$. The inclusions $\C[\mathfrak{h}]^W \hookrightarrow Z_{\mathbf{c}}(W)$ and $\C[\mathfrak{h}^*]^W \hookrightarrow Z_{\mathbf{c}}(W)$ define surjective morphisms
\bdm
\pi_1 \, : \, X_{\mathbf{c}}(W) \twoheadrightarrow \h^*/W \qquad \textrm{ and } \qquad \pi_2 \, : \, X_{\mathbf{c}}(W) \twoheadrightarrow \h/W.
\edm
We write 
\bdm
\Upsilon \,  : \,  X_{\mathbf{c}}(W,\mathfrak{h}) \twoheadrightarrow \h^* / W \times \h/W
\edm
for the product morphism $\Upsilon = \pi_1 \times \pi_2$. It is a finite, and hence closed, surjective morphism.

\subsection{Parabolic subgroups}\label{sec:para}
Let $W'$ be a subgroup of $W$. It is called a \textit{parabolic subgroup} if there is a set $S \subseteq \mathfrak{h}$ such that $W' = \textrm{Stab}_{\, W} (S)$. Since $W$ acts linearly on $\mathfrak{h}$ every parabolic subgroup is the stablizer of some linear subspace of $\mathfrak{h}$. By a theorem of Steinberg \cite[Theorem 1.5]{8}, a parabolic subgroup is itself a complex reflection group. Note that, in general, there exist subgroups of $W$ that are themselves complex reflection groups but are not parabolic subgroups e.g. $\Z / 2 \Z \subset \Z / 4 \Z$. The result \cite[Proposition 1.10]{Humphreys} shows that this behaviour does not happen in Weyl groups. We write
\bdm
(\mathfrak{h}^{*W'})^\perp := \{ y \in \mathfrak{h} \, | \, x(y) = 0 \, \textrm{ for all } \, x \in \mathfrak{h}^{*W'} \}.
\edm 
Then $\mathfrak{h} = \mathfrak{h}^{W'} \oplus (\mathfrak{h}^{*W'})^\perp$ is a decomposition of $\mathfrak{h}$ as a $W'$-module. Define the rank of a complex reflection group $W'$ to be the dimension of a faithful reflection representation of $W'$ of minimal rank. Note that $(\mathfrak{h}^{*W'})^\perp$ is a faithful reflection representation of $W'$ of minimal rank hence the rank of $W'$ is $\dim (\mathfrak{h}^{*W'})^\perp$. When $W$ is a real reflection group this definition of rank agrees, by \cite[Theorem 1.12]{Humphreys}, with the alternative definition of rank in terms of root systems (\cite[1.3]{Humphreys}). The group $W$ acts on its set of parabolic subgroups by conjugation. Given a parabolic subgroup $W'$ the corresponding conjugacy class will be denoted $(W')$. We also require the partial ordering on conjugacy classes of parabolic subgroups of $W$ defined by $(W_1) \ge (W_2)$ if and only if $W_1$ is conjugate to a subgroup of $W_2$ (the ordering is choosen in this way so that it agrees with a geometric ordering to be introduced in Section \ref{sec:Poisson}). Finally, for a given parabolic subgroup $W'$ of $W$, we denote by $\mathfrak{h}^{W'}_{\textrm{reg}}$ the subset of $\h^{W'}$ consisting of those points whose stabliser is $W'$: it is a locally closed subset of $\h$.

\section{Complete Poisson algebras}\label{sec:Poisson}

\subsection{} In this section we state and prove certain results on completed Poisson algebras that are required but that the author was unable to find suitible references for.

\subsection{Poisson Ideals}\label{lem:dim} Throughout $R$ will denote a commutative, affine domain over a field $k$. If $I$ is a proper ideal of $R$ then Krull's Intersection Theorem (\cite[Corollary 5.4]{Eisenbud}) says that
\bdm
\bigcap_{n = 1}^\infty I^n = 0.
\edm
Therefore, if $\widehat{R}_I$ denotes the completion of $R$ along $I$, the natural map $j : R \rightarrow \widehat{R}_I$ is an inclusion. The Krull dimension of $R$ will be written $\Kldim R$.

\begin{lem}
For $R$, $I$ as above,
\bdm
\Kldim R = \Kldim \widehat{R}_I
\edm
\end{lem}

\begin{proof}
Let $\mathfrak{n}$ be a maximal ideal of $\widehat{R}_I$, then \cite[Corollary 2.19]{GS} shows that $\mathfrak{n} \mapsto \mathfrak{n} \cap R$ defines a bijection between the maximal ideals of $\widehat{R}_I$ and the maximal ideals of $R$ containing $I$. Moreover, the proof of \cite[Theorem 7.5]{GS} says that $\height (\mathfrak{n}) = \height (\mathfrak{n} \cap R)$. Therefore $\Kldim \widehat{R}_I = \sup \{ \height (\mathfrak{m}) \}$, where $\mathfrak{m}$ ranges over all maximal ideals of $R$ that contain $I$. Since $R$ is an affine domain over $k$, \cite[Theorem A]{Eisenbud}) says that $\height (\mathfrak{m}) = \Kldim R$ for all maximal ideals of $R$, hence $\Kldim R = \Kldim \widehat{R}_I$. 
\end{proof}

\subsection{}\label{lem:primes} It will be particularly important for us later to understand what happens to prime ideals when passing to completions.

\begin{lem}
Choose a prime ideal $P \triangleleft R$ such that $P \otimes_R \widehat{R}_I \neq \widehat{R}_I$ and $Q$ a prime ideal of $\widehat{R}_I$. Then
\begin{enumerate}

\item For each prime $Q'$ minimal over $P \otimes_R \widehat{R}_I$, $\height (Q') = \height (P)$ and $Q' \cap R = P$. 

\item $Q \cap R$ is a prime ideal and $\height (Q) = \height (Q \cap R)$.

\item If $I \subseteq P$ then $P \otimes_R \widehat{R}_I$ is prime in $\widehat{R}_I$.

\end{enumerate}
\end{lem}

\begin{proof}
Clearly $Q \cap R$ is a prime ideal. By \cite[Theorem 7.2]{Eisenbud}, $\widehat{R}_I$ is a flat extension of $R$ therefore \cite[Lemma 10.11]{Eisenbud} shows that (Going down) holds. Now let $Q'$ be a prime minimal over $P \otimes_R \widehat{R}_I$. If $Q' \cap R \neq P$ then by (Going down) there exists a prime $Q_0 \subsetneq Q'$ such that $Q_0 \, \cap \, R = P \subsetneq Q' \, \cap \, R$. But then $P \otimes_R \widehat{R}_I \subset Q_0$, contradicting the minimality of $Q'$. Fix a maximal chain of primes $P_0 \supset P_1 \supset \dots \supset P_n = 0$ such that $P_i = Q \cap R$ and $I \subseteq P_0$. By \cite[Theorem A, page 286]{Eisenbud}, $R$ is universally caternary hence $n = \Kldim R$. The result \cite[Corollary 2.19]{GS} says that there is a unique maximal ideal $\mathfrak{n} =: Q_0$ of $\widehat{R}_I$ such that $\mathfrak{n} \cap R = P_0$. The proof of Lemma \ref{lem:dim} shows that $\Kldim R = \height (\mathfrak{m}) =  \height (\mathfrak{n}) =  \Kldim \widehat{R}_I$. Applying (Going down) to $P_0 \supset P_1$ shows that there exists a prime $Q_1$ such that $Q_1 \cap R = P_1$ and $Q_1 \subsetneq Q_0$. Clearly $\height (P_1) \ge \height (Q_1)$. By repeating this argument we get a chain of primes $Q_0 \supset Q_1 \supset \dots \supset Q_n$ such that $Q_i \cap R = P_i$ and $\height (P_i) \ge \height (Q_i)$. But Lemma \ref{lem:dim} implies that we must have $\height (Q_i) = \height (P_i)$. In particular, $\height (Q) = \height (Q \cap R)$. This completes the proof of $(1)$ and $(2)$. \\
By \cite[Theorem 7.2]{Eisenbud}, $\widehat{P} :=  P \otimes_R \widehat{R}_I = \lim_{\infty \leftarrow n} \, P / I^n$ (note that $I \subset P$ implies $\widehat{P} \neq \widehat{R}_I$). Let us show that $\widehat{P}$ is prime. If not then there exist $a,b \in \widehat{R}_I \backslash \widehat{P}$ such that $a \cdot b \in \widehat{P}$. Therefore there exists some $N \>> 0$ such that $\bar{a},\bar{b} \in (R/I^N) \, \backslash \, (P/I^N)$ with $\bar{a} \cdot \bar{b} \in P/I^N$. But this is a contradiction since $P/I^N$ is prime.
\end{proof}

\subsection{}\label{lem:polyprimes} If $S_1$ and $S_2$ are $k$-algebras, complete with respect to the ideals $I_1$ and $I_2$ respectively then the completed tensor product is defined to be
\bdm
S_1 \, \widehat{\otimes} \, S_2 := \lim_{\infty \leftarrow n} (S_1 \otimes S_2) / J^n,
\edm
where $J := I_1 \otimes S_2 + S_1 \otimes I_2$.

\blem
Let $P$ be a prime ideal of $\widehat{R}_I$ and $Q$ the ideal generated by $P$ in $\widehat{R}_I \, \widehat{\otimes} \, k[[x]]$. Then $Q$ is prime.
\elem

\begin{proof}
Since $\widehat{R}_I$ is Noetherian, the ideal $P$ is finitely generated. By \cite[Theorem 7.2]{Eisenbud}, 
\bdm
Q = \lim_{\infty \leftarrow n} \, P \otimes k[x] / J^n = (P \otimes k[x]) \otimes_{\widehat{R}_I \,  \otimes \, k[x]} \, \widehat{R}_I \, \widehat{\otimes} \, k[[x]] = \{ \, \sum_{i \ge 0} p_i x^i \, | \, p_i \in P \, \},
\edm
is a finitely generated ideal in $\widehat{R}_I \, \widehat{\otimes} \, k[[x]]$, where $J = I \otimes k[x] + R \otimes (x)$. Now choose $a = \sum_{i \ge 0} a_i x^i, b = \sum_{j \ge 0 } b_j x^j \in \widehat{R}_I \, \widehat{\otimes} \, k[[x]]$ such that $a \cdot b \in Q$. If $a,b \notin Q$ then we can choose $r,s \in \N$ to be minimal with respect to the properties $a_r,b_s \notin P$. Then the fact that the coefficient of $x^{r +s}$ in the expansion of $a \cdot b$ lies in $P$ is a contradiction.
\end{proof}

\subsection{}\label{lem:Poisson} For the reminder of this section we make the additional assumptions that $R$ is a Poisson algebra with bracket $\{ \cdot , \cdot \}$ and that $k = \C$. An ideal $I$ of $R$ is said to be a \textit{Poisson ideal} if $\{ I, R \} \subset I$. A prime ideal that is Poisson is simply called a Poisson prime.  

\begin{lem}
Let $R$, $I$ be as above. We do not assume that $I$ is a Poisson ideal. 
\begin{enumerate}

\item $\widehat{R}_I$ is a Poisson algebra.

\item If $Q$ is a Poisson prime of $\widehat{R}_I$ then $Q \cap R$ is a Poisson prime ideal.

\item If $J$ is a Poisson ideal such that $J \otimes_R \widehat{R}_I \neq \widehat{R}_I$ then $J \otimes_R \widehat{R}_I$ is a Poisson ideal and any prime minimal over $J \otimes_R \widehat{R}_I$ is Poisson.

\end{enumerate}
\end{lem}

\begin{proof}
Each element of $\widehat{R}_I$ has the form $(f_i)_{i \in \N}$, where $f_i \in R/I^i$ and $f_j \equiv f_i \, \mod \, I^i$ for all $j > i$. The Poisson structure on $\widehat{R}_I$, (denoted $\langle \cdot , \cdot \rangle$) is defined as $\langle f, g \rangle_i := \{ f_{i+1},g_{i+1} \} + I^i$ (alternatively one can simply note that, for fixed $f \in R$, $\{ f , \, - \, \}$ is a derivation of $R$ and thus continuous in the $I$-adic topology). Denote by $\iota : R \hookrightarrow \widehat{R}_I$ the inclusion map. Let $f,g \in R$, then $\langle \iota (f), \iota(g) \rangle_i = \langle f + I^{i+1}, g + I^{i+1} \rangle_i = \{ f , g \} + I^i$. Therefore $\langle \iota (f), \iota (g) \rangle = \iota ( \{ f, g \})$ and $(2)$ follows from this.\\
\noindent To show that $J \otimes_R \widehat{R}_I$ is a Poisson ideal, choose $(f_i)_{i \in \N} \in J \otimes_R \widehat{R}_I$ and $(g_i)_{i \in \N} \in \widehat{R}_I$. Then, for each $i$ in $\N$, there exists $p_i \in J$ such that $p_i \equiv f_i \, \mod \, I^i$ and $\langle (f_i),(g_i) \rangle_i = \{ f_{i+1},g_{i+1} \} + I^i = \{ p_{i+1} , g_{i+1} \} + I^i \in (J + I^i)/I^i$. Hence $\langle J \otimes_R \widehat{R}_I, \widehat{R}_I \rangle \subset J  \otimes_R \widehat{R}_I$. Noting that $k = \C$,  \cite[Lemma 3.3.3]{Dixmier} says that the primes minimal over $J \otimes_R \widehat{R}_I$ are Poisson.
\end{proof} 

\subsection{}\label{lem:Poissonprimitive} Following \cite[Section 3.2]{Poisson Orders}, we define the \textit{Poisson core} of an ideal $J$ of $R$ to be the largest Poisson ideal of $R$ contained in $J$ and denoted it $\mathcal{C}(J)$. It exists because the sum of two Poisson ideals is again a Poisson ideal. If $J$ is prime then $\mathcal{C}(J)$ is also prime and when $\mathfrak{m}$ is maximal, $\mathcal{C}(\mathfrak{m})$ is said to be \textit{Poisson primitive}. We say that $\mf{m}$ is \textit{maximal and Poisson} if it is a maximal ideal of $R$ that is Poisson. Clearly, every maximal and Poisson ideal is Poisson primitive.

\begin{lem}
Let $R$ and $I$ be as above and choose $\mathfrak{m}$ a maximal ideal of $R$ containing $I$. Then every prime minimal over $\mathcal{C}(\mathfrak{m}) \otimes_R \widehat{R}_I$ is Poisson primitive and the Poisson core of $\mathfrak{m} \otimes_R \widehat{R}_I$ is one of these minimal primes. Conversely, if $J$ is a Poisson primitive ideal in $\widehat{R}_I$ then $J \cap R$ is Poisson primitive.
\end{lem}

\begin{proof}
By \cite[Corollary 2.19]{GS}, $I \subset \mathfrak{m}$ implies that $\widehat{R}_I \neq \mathfrak{m} \otimes_R \widehat{R}_I $ is a maximal ideal of $\widehat{R}_I$. Therefore $\mathcal{C}(\mathfrak{m}) \otimes_R \widehat{R}_I$ is also a proper ideal of $\widehat{R}_I$, which is Poisson by Lemma \ref{lem:Poisson}. Let $P$ be a prime minimal over $\mathcal{C}(\mathfrak{m}) \otimes_R \widehat{R}_I$. Again by Lemma \ref{lem:Poisson}, it is Poisson. Since \cite[Corollary 2.19]{GS} says that there is a bijection between maximal ideals of $\widehat{R}_I$ and maximal ideals of $R$ containing $I$ it suffices to consider the case $P \subseteq \mathfrak{m} \otimes_R \widehat{R}_I$. If $\mathcal{C}(\mathfrak{m}) = \mathfrak{m}$ then the result is trivial therefore, without loss of generality, $\mathcal{C}(\mathfrak{m}) \subsetneq \mathfrak{m}$. Assume that $P$ is not the Poisson core of $\mathfrak{m} \otimes_R \widehat{R}_I$, so that $P \subsetneq Q = \mathcal{C}(\mathfrak{m} \otimes_R \widehat{R}_I) \subseteq \mathfrak{m} \otimes_R \widehat{R}_I$. By Lemma \ref{lem:primes}, $\mathcal{C}(\mathfrak{m}) = R \cap P \subseteq Q \cap R \subseteq \mathfrak{m} \otimes_R \widehat{R}_I \cap R = \mathfrak{m}$, and Lemma \ref{lem:Poisson} says that $Q \cap R$ is a Poisson prime. Therefore $Q \cap R = \mathcal{C}(\mathfrak{m})$ by maximality. But Lemma \ref{lem:primes} says that
\begin{displaymath}
\height \mathcal{C}(\mathfrak{m}) = \height (P) < \height (Q) = \height (Q \cap R).
\end{displaymath}
This contradiction shows that $P$ is Poisson primitive. The same argument also implies the converse statement.
\end{proof}

\subsection{}\label{sec:Poissondefinition} Now let $\msf{A}$ be a $\C$-algebra, $\mbf{t}$ a central non-zero divisor and $\rho \, : \, \msf{A} \twoheadrightarrow A := \msf{A} / \mbf{t} \cdot \msf{A}$ the quotient map. Assume that there exists an affine central subalgebra $Z$ of $A$ such that $A$ is a finite module over $Z$. Let $\{ z_i \, : \, i \in I \}$ be a $\C$-basis for $Z$ and choose a lift $\hat{z}_i$ of $z_i$ in $\msf{A}$ for every $i \in I$. As noted in \cite[(2.2)]{Poisson Orders}, the rule
\beq\label{eq:Poissonbraket}
\{ z_i, z_j \} = \rho([ \hat{z}_i, \hat{z}_j] / \mbf{t})
\eeq
extends by linearity to a Poisson bracket on $Z$. The bracket is independent of the choice of lifts $\hat{z}_i$. If $a \in A$ and we choose a lift $\hat{a}$ of $a$ in $\msf{A}$ then equation (\ref{eq:Poissonbraket}) defines an action of $Z$ on $A$, $z_i \cdot a := \rho([ \hat{z}_i, \hat{a}] / \mbf{t})$, making $A$ into a Poisson module for $Z$.\\

\subsection{}\label{sec:2algebras} For $i = 1,2$ we choose $\msf{A}_i$ to be a $\C$-algbera, $\mbf{t}_i \in \msf{A}_i$ a central non-zero divisor and $\rho_i \, : \, \msf{A}_i \twoheadrightarrow A_i := \msf{A}_i / \mbf{t}_i \msf{A}_i$. Assume that there exists a finite dimensional, abelian Lie subaglebra $\mf{n}_i$ of $\msf{A}_i$ such that the adjoint action of $\mf{n}_i$ on $\msf{A}_i$ is locally nilpotent. Denote by $\Ui$ the associative subalgebra (without unit) in $\msf{A}_i$ generated by $\mf{n}_i$ and let $\Ui^k$ be the $k^{th}$ power of $\Ui$ ($k \in \N$). As noted in \cite[(5.1)]{Primitive}, for any $a \in \msf{A}_i$ there exists $n \in \Z$ (depending on $a$) such that
\beq\label{eq:nilpotent}
a \cdot \Ui^k \subset \Ui^{k + n} \cdot \msf{A}_i \qquad \forall k \gg 0.
\eeq
We make the additional assumption that the image of $\mf{n}_i$ under $\rho_i$ is contained in the centre $Z_i$ of $A_i$. The ideal generated in $Z_i$ by $\rho_i(\mf{n}_i)$ will be denoted $I_i$. We assume that $Z_i$ is affine and $A_i$ a finite module over $Z_i$. Property (\ref{eq:nilpotent}) implies that the space
\bdm
\widehat{\msf{A}}_i := \lim_{\infty \leftarrow k} \msf{A}_i \, / \, \Ui^k \cdot \msf{A}_i, \quad i = 1,2
\edm
is an associative algebra that is complete with respect to the topology on $\msf{A}_i$ defined by the set $\{ \Ui^k \cdot \msf{A}_i \}_{k \ge 1}$ of fundamental neighborhoods of zero. 

\subsection{}\label{lem:centres} Finally, we assume that there exists an isomorphism
\bdm
\theta \, : \, \widehat{\msf{A}}_1 \stackrel{\sim}{\longrightarrow} \widehat{\msf{A}}_2
\edm
such that $\theta(\mbf{t}_1) = \mbf{t}_2$ and $\theta( \, \Uone^k \cdot \widehat{\msf{A}}_1) = \Utwo^k \cdot \widehat{\msf{A}}_2$ for all $k \ge 0$ (thus $\theta$ is a homeomorphism). We write $ \widehat{A}_i :=  \widehat{\msf{A}}_i \, / \, \mbf{t}_i \cdot  \widehat{\msf{A}}_i$ and let $ \widehat{Z}_i$ be the completion of $Z_i$ with respect to the ideal $I_i$.

\begin{lem}
Let $\msf{A}_i$, $\Ui$, $Z_i$ and $I_i$ be as above. Then
\bdm
Z( \widehat{A}_i) =  \widehat{Z}_i.
\edm
\end{lem}

\begin{proof}
Since $Z_i$ is a Noetherian ring, $\widehat{Z}_i$ is a flat $Z_i$-module and $\widehat{A}_i = A_i \otimes_{Z_i} \widehat{Z}_i$. We choose a generating set $a_1, \ds, a_n$ of $A_i$ as a module over $Z_i$ and assume without loss of generality that $a_1 = 1$. The flatness of $\widehat{Z}_i$ implies that the natural map $\widehat{Z}_i \rightarrow \widehat{A}_i$ is an embedding. Its image is central, therefore it suffices to show that $Z(\widehat{A}_i) \subseteq \widehat{Z}_i$. Let $h$ be central in $\widehat{A}_i$. We prove by induction on $1 \le l \le n$ that there exist $h_j \in A$ and $z_j \in \widehat{Z}_i$ such that $h = \sum_j h_j \otimes z_j$ and the $h_j$'s commute with every $a_t$, $t \le l$. This is clear when $l = 1$. Therefore assume $l >1$ and that there exist $h_j,z_j$ such that $h = \sum_j h_j \otimes z_j$ and the $h_j$'s commute with all $a_t, \, t < l$. Since $\sum_j [a_l,h_j] \otimes z_j = 0$, the flatness of $\widehat{Z}_i$ implies that there exist $b_{jk} \in Z_i$ and $z_k' \in \widehat{Z}_i$ such that 
\begin{enumerate}
\item $\sum_{k} b_{jk} z_k' = z_j$ in $\widehat{Z}_i$,
\item $\sum_j [a_l,h_j] b_{jk} = 0$ in $A_i$ i.e. $[a_l,\sum_{j} h_j b_{jk}] = 0$.
\end{enumerate}    
Therefore $h_k' := \sum_j h_j b_{jk}$ commutes with $a_1, \ds, a_{l-1}, a_l$. However $(1)$ also implies that $h = \sum_k h_k' \otimes z_k'$. Therefore induction implies that $h \in \widehat{Z}_i$.
\end{proof}

\begin{prop}\label{prop:poissoniso}
Assume that $Z_i$ is a direct summand of $A_i$ as a $Z_i$-module. The isomorphism $\theta$ induces a \textbf{Poisson} isomorphism 
\bdm
\theta \, : \, \widehat{Z}_1 \stackrel{\sim}{\longrightarrow} \widehat{Z}_2
\edm
\end{prop}

\begin{proof}
Since $\theta(\mbf{t}_1) = \mbf{t}_2$, $\theta$ defines an isomorphism $\widehat{A}_1 \stackrel{\sim}{\longrightarrow} \widehat{A}_2$. This restricts to an isomorphism of the centres. By Lemma \ref{lem:centres}, $Z( \widehat{A}_i) =  \widehat{Z}_i$, and $\theta$ induces an isomorphism $\widehat{Z}_1 \stackrel{\sim}{\longrightarrow} \widehat{Z}_2$. Therefore we must show that $\theta$ is a Poisson morphism. Let $u,v \in \widehat{Z}_1$, $u = (u_i)_{i \ge 0}$ and $v = (v_i)_{i \ge 0}$ where $u_i,v_i \in Z_1 \, / \, I_1^i$ and choose lifts of $u,v$ to $\hat{u}$ and $\hat{v}$ in $\msf{A}_1$. The fact that $\theta$ induces an isomorphism $\widehat{Z}_1 \cong \widehat{Z}_2$ together with the fact that $\theta \circ \rho_1 = \rho_2 \circ \theta$ (since $\theta(\mbf{t}_1) = \mbf{t}_2$) imply that $\theta( \hat{u})$ is a lift of $\theta(u)$. The assumption that $Z_i$ is a direct summand of $A_i$ as a $Z_i$-module implies that $Z_i \cap \, (  \Ui^k \cdot A_i ) = I_i^k$ and hence
\bdm
Z_i \, / \, I^k_i \hookrightarrow A_i \, / \, \Ui^k \cdot A_i \quad \forall \, k \ge 0.
\edm
We recall the definition of the Poisson braket on $\widehat{Z}_i$ (combining Lemma \ref{lem:Poisson} and equation (\ref{eq:Poissonbraket})):
\bdm
(\{ u,v \})_i := \rho_1( [\hat{u}_{i+1},\hat{v}_{i+1}] / \mbf{t}_1 ) \, \mod \, I_1^i.
\edm
Now
\bdm
\begin{array}{rcl}
( \theta( \{ u,v \} ))_i & = & \theta( \rho_1( [\hat{u}_{i+1},\hat{v}_{i+1}] / \mbf{t}_1 ) \, \mod \, I_1^i) \\
 & = &  \theta( \rho_1( [\hat{u}_{i+1},\hat{v}_{i+1}] / \mbf{t}_1 ) \, \mod \, \Uone^i \cdot A_1) \\
 & = &  \theta( \rho_1( [\hat{u}_{i+1},\hat{v}_{i+1}] / \mbf{t}_1 ) ) \, \mod \, \Utwo^i \cdot A_2\\
 & = &  \rho_2 ( \theta( [\hat{u}_{i+1},\hat{v}_{i+1}]) / \mbf{t}_2 ) ) \, \mod \, \Utwo^i \cdot A_2\\
 & = &  \rho_2 ( [\theta( \hat{u}_{i+1}),\theta( \hat{v}_{i+1}) ] / \mbf{t}_2 ) ) \, \mod \, \Utwo^i \cdot A_2\\
 & = &  \rho_2 ( [\theta( \hat{u}_{i+1}),\theta( \hat{v}_{i+1}) ] / \mbf{t}_2 ) ) \, \mod \, I^i_2\\
 & = & ( \{ \theta(u) ,\theta(v) \} )_i,
\end{array}
\edm
where in the second and sixth line we have used the fact that $Z_i \, / \,  I^k_i \hookrightarrow A_i \, / \, \Ui^k \cdot A_i$, in the fourth line we use the fact that $\theta \circ \rho_1 = \rho_2 \circ \theta$ and in the final line we use the fact that $\theta(\hat{u})$ is a lift of $\theta(u)$ to $\msf{A}_2$.  
\end{proof}

\section{Completions of the generalised Calogero-Moser Space} 

\subsection{} In the remainder of the article we wish to consider rational Cherednik algebras associated to the same complex reflection group but with different reflection representations. Therefore, to avoid any ambiguities, we will write $H_{\mbf{c}}(W,\h)$, $Z_{\mbf{c}}(W,\h)$, $X_{\mbf{c}}(W,\h)$ and so on to keep track of this additional information. We can consider the rational Cherednik algebra $H_{\mbf{t},\mbf{c}}(W,\h)$, where $\mbf{t}$ is a central indeterminate. It is a $\C [\mbf{t}]$-algebra and there is a canonical isomorphism 
\bdm
\rho \, : \, H_{\mbf{t},\mbf{c}}(W,\h) / \mbf{t} \cdot H_{\mbf{t},\mbf{c}}(W,\h) \stackrel{\sim}{\rightarrow} H_{0,\mbf{c}}(W,\h).
\edm
Since the centre $Z_{\mbf{c}}(W,\h)$ of $H_{0,\mbf{c}}(W,\h)$ is an affine domain over which $H_{0,\mbf{c}}(W,\h)$ is a finite module we are in the situation described in (\ref{sec:Poissondefinition}). Hence $Z_{\mbf{c}}(W,\h)$ is a Poisson algebra. If $X_{\mbf{c}}(W,\h)$ is considered as a (non-smooth) complex analytic Poisson manifold then it is stratified by symplectic leaves, which are the maximal connected complex analytic submanifolds of $X_{\mbf{c}}(W,\h)$ on which the bracket $\{ -, - \}$ is nondegenerate. It was shown in \cite[Theorem 7.8]{Poisson Orders} that the symplectic leaves of $X_{\mbf{c}}(W,\h)$ are algebraic and there are only finitely many. Here algebraic means that the closure of a leaf $\mc{L}$ is an irreducible algebraic subset of $X_{\mbf{c}}(W,\h)$ and $\mc{L}$ is a Zariski open subset of its closure. In particular, the closure of $\mc{L}$ is defined by a Poisson primitive ideal.  

\subsection{} The polynomial ring $\C[\mathfrak{h}/W]$ is generated by the vector space of linear functionals $(\h / W)^*$. Let $b \in \mathfrak{h}$ and $\lambda \in (\h / W)^*$. We can evaluate $\lambda$ on the orbit $W \cdot b$, $b \mapsto \lambda(b)$. Let $\mathfrak{m}(b) := \{ \lambda - \lambda(b) \, | \, \lambda \in (\h / W)^* \}$. The ideal generated by $\mathfrak{m}(b)$ in $\C[\mathfrak{h}/W]$ is the maximal ideal corresponding to the orbit $W \cdot b \in \h / W$. Similarly, if $W_b$ is the stabilizer of $b$ in $W$, let $\mathfrak{n}(q) := \{ \lambda - \lambda(q) \, | \, \lambda \in (\h / W_b)^* \}$ for each $q \in \h$. As noted in \cite[Section 6]{Primitive}, we are in the setup of (\ref{sec:2algebras}) if we take $\msf{A}_1 = H_{\mbf{t},\mbf{c}}(W,\h)$, $\mf{n}_1 = \mf{m}(b)$, $\msf{A}'_2 = H_{\mbf{t},\mbf{c}'}(W_b,\h)$ and $\mf{n}'_2 = \mf{n}(0)$. Thus we get complete, assocaitive algebras
\bdm
\widehat{H}_{\mbf{t},\mathbf{c}}(W,\h)_b := \lim_{\infty \leftarrow k} H_{\mbf{t},\mathbf{c}}(W,\h) / \mf{m}(b)^k \cdot H_{\mbf{t},\mathbf{c}}(W,\h),
\edm
\bdm
\widehat{H}_{\mbf{t},\mathbf{c}'}(W_b,\h)_0 := \lim_{\infty \leftarrow k} H_{\mbf{t},\mathbf{c}'}(W_b,\h) / \mf{n}(0)^k \cdot H_{\mbf{t},\mathbf{c}'}(W_b,\h).
\edm
To get $\msf{A}_2$, $\mf{n}_2$ and $\theta$ we need to introduce a certain centralizer algebra.

\subsection{Centralizer algebras}\label{thm:BEiso}
We recall the centralizer construction described in \cite[3.2]{BE}. Let $A$ be a $\C$-algebra equipped with a homomorphism $H \longrightarrow A^{\times}$, where $H$ is a finite group. Let $G$ be another finite group such that $H$ is a subgroup of $G$. The algebra $C(G,H,A)$ is defined to be the centralizer of $A$ in the right $A$-module $P := \Fun_H(G,A)$ of $H$-invariant, $A$-valued functions on $G$. By making a choice of left coset representatives of $H$ in $G$, $C(G,H,A)$ is realized as the algebra of $|G/H|$ by $|G/H|$ matrices over $A$. Let $\msf{A}_2 = C(W,W_b,\widehat{H}_{\mbf{t},\mathbf{c}'}(W_b,\mathfrak{h})_0)$ and $\mf{n}_2 = C(W,W_b,\mf{n}(0))$.
  
\begin{thm}[\cite{BE}, Theorem 3.2]
Let $b \in \mathfrak{h}$, and define $\mathbf{c}'$ to be the restriction of $\mathbf{c}$ to the set $S_b$ of reflections in $W_b$. Then one has an isomorphism of $\C[\mbf{t}]$-algebras
\beq\label{eq:BEiso}
\theta : \widehat{H}_{\mbf{t},\mathbf{c}} (W,\mathfrak{h})_b \rightarrow C(W,W_b,\widehat{H}_{\mbf{t},\mathbf{c}'}(W_b,\mathfrak{h})_0),
\eeq
defined by the following formulas. Suppose that $f \in \textrm{Fun}_{W_b}(W,\widehat{H}_{t,\mathbf{c}'}(W_b,\mathfrak{h})_0)$. Then
\begin{displaymath}
(\theta(u)f)(w) = f(wu),u \in W;
\end{displaymath}
for any $\alpha \in \mathfrak{h}^*$,
\begin{displaymath}
(\theta(x_{\alpha})f)(w) = (x_{w\alpha}^{(b)} + (w\alpha,b))f(w),
\end{displaymath}
where $x_{\alpha} \in \mathfrak{h}^* \subset H_{t,\mathbf{c}}(W,\mathfrak{h}),x_{w\alpha}^{(b)} \in H_{t,\mathbf{c}'}(W_b,\mathfrak{h})$; and for any $a \in \mathfrak{h}$,
\begin{displaymath}
(\theta(y_a)f)(w) = y_{wa}^{(b)}f(w) + \sum_{s \in S:s \notin W_b} \frac{2c_s}{1 - \lambda_s} \frac{\alpha_s(wa)}{x_{\alpha_s}^{(b)} + \alpha_s(b)}(f(sw) - f(w)).
\end{displaymath}
where $y_a \in \mathfrak{h} \subset H_{t,\mathbf{c}}(W,\mathfrak{h})$ and $y_a^{(b)}$ the same vector considered now as an element of $H_{t,\mathbf{c}'}(W_b,\mathfrak{h})$.
\end{thm}

\subsection{}\label{lem:auto} Choose homogeneous, algebraically independent generators $F_1, \dots F_n$ of $\C[\mathfrak{h}]^W$ and $P_1, \dots , P_n$ of $\C[\mathfrak{h}]^{W_b}$. 

\begin{lem}[Lemma 3.1, \cite{Factor}]
For each $b \in \mathfrak{h}$ the map $\Psi : \C[[\mathfrak{h}/W_b]]_0 \longrightarrow \C[[\mathfrak{h}/W_b]]_0$ defined by
\begin{displaymath}
P_i(\mathbf{x}) \mapsto F_i(\mathbf{x} + b) - F_i(b)
\end{displaymath}
is an automorphism.
\end{lem}

\begin{prop}\label{prop:quoiso}
Let $\theta : \widehat{H}_{t,\mathbf{c}} (W,\mathfrak{h})_b \rightarrow C(W,W_b,\widehat{H}_{t,\mathbf{c}'}(W_b,\mathfrak{h})_0)$ be the isomorphism (\ref{eq:BEiso}). Then 
\bdm
\theta( \mathfrak{m}(b)^k \cdot H_{t,\mathbf{c}}(W,\mathfrak{h}) ) = C \left( W,W_b, \mathfrak{n}(0) ^k \cdot H_{t,\mathbf{c}'}(W_b,\mathfrak{h}) \right), \qquad \forall k \ge 1. 
\edm
\end{prop}

\begin{proof}
This is a modification of the proof of \cite[Corollary 3.2]{Factor}, which is the above result in the special case $k = 1$. It is shown in the proof of \textit{loc. cit.} that if $g \in \mathfrak{m}(b)^k \lhd \C[\mathfrak{h}]^W$, then $g(\mathbf{x} + b) \in \mathfrak{n}(0)^k \lhd \C[\mathfrak{h}]^{W_b}$. This shows that $\theta(g)f(w) \in \mathfrak{n}(0)^k \cdot \widehat{H}_{t,\mathbf{c}'}(W_b)_0$ and 
\begin{equation}\label{eq:inc}
\theta(\mathfrak{m}(b)^k \cdot \widehat{H}_{t,\mathbf{c}} (W,\mathfrak{h})_b) \subseteq C(W,W_b,\mathfrak{n}(0)^k \cdot \widehat{H}_{t,\mathbf{c}'}(W_b,\mathfrak{h})_0).
\end{equation}
The ideal $\mathfrak{m}(b)$ in $\C[\mathfrak{h}]^W$ is generated by $F_1(\mathbf{x}) - F_1(b), \dots , F_n(\mathbf{x}) - F_n(b)$ and we have $\theta(F_i(\mathbf{x}) - F_i(b))f(w) = (F_i(\mathbf{x} + b) - F_i(b))f(w)$. The statement of Lemma \ref{lem:auto} is equivalent to the fact that 
\begin{displaymath}
\{ F_1(\mathbf{x} + b) - F_1(b), \dots , F_n(\mathbf{x} + b) - F_n(b)) \} \cdot \C[[\mathfrak{h}/W_b]]_0 = \mathfrak{n}(0) \cdot \C[[\mathfrak{h}/W_b]]_0,
\end{displaymath}
which in turn implies that 
\bdm
\{ F_1(\mathbf{x} + b) - F_1(b), \dots , F_n(\mathbf{x} + b) - F_n(b)) \}^k \cdot \C[[\mathfrak{h}/W_b]]_0 = \mathfrak{n}(0)^k \cdot \C[[\mathfrak{h}/W_b]]_0.
\edm
This, together with (\ref{eq:inc}), implies that 
\begin{displaymath}
\theta(\mathfrak{m}(b)^k \cdot \widehat{H}_{t,\mathbf{c}} (W,\mathfrak{h})_b) = C(W,W_b,\mathfrak{n}(0)^k \cdot \widehat{H}_{t,\mathbf{c}'}(W_b,\mathfrak{h})_0).
\end{displaymath}
\end{proof}

\subsection{}\label{thm:completeiso} Let us denote by $\widehat{Z}_{\mbf{c}}(W,\h)$ the completion of $Z_{\mbf{c}}(W,\h)$ with respect to the ideal generated by $\mf{m}(b)$. Similarly, let $\widehat{Z}_{\mbf{c}'}(W_b,\h)_0$ be the completion of $Z_{\mbf{c}'}(W_b,\h)$ with respect to the ideal generated by $\mf{n}(0)$. Lemma \ref{lem:centres} says that
\bdm
Z(\widehat{H}_{0,\mathbf{c}} (W,\mathfrak{h})_b) = \widehat{Z}_{\mbf{c}}(W,\h) \quad \textrm{ and } \quad Z(C(W,W_b,\widehat{H}_{0,\mathbf{c}'}(W_b,\mathfrak{h})_0)) = \widehat{Z}_{\mbf{c}'}(W_b,\h)_0.
\edm

\begin{lem}
The centre $Z_{\mbf{c}}(W,\h)$ of $H_{\mbf{c}}(\h,W)$ is a direct summand of $H_{\mbf{c}}(\h,W)$ when considered as a $Z_{\mbf{c}}(W,\h)$-module.
\end{lem}

\begin{proof}
First, let us show that $Z_{\mbf{c}}(W,\h)$ is integrally closed. The (Zariski closed) set of points where the group $W$ does not act freely on $\h \times \h^*$ has codimension at least two. Then \cite[Theorem 4.6]{skewmaximal} says that the skew group ring $\C[\h \oplus \h^*] \rtimes W$ is a maximal order. The algebra $H_{\mbf{c}}(\h,W)$ is $\N$-filtered and $\C[\h \oplus \h^*] \rtimes W$ is its associated graded. Now \cite[Theorem 5]{filteredmaxorder} shows that the property of being a maximal order lifts to $H_{\mbf{c}}(\h,W)$. The centre of a maximal order is integrally closed, see \cite[Proposition 5.1.10]{MR}.\\
The statement of the Lemma now follows from:\\
 \noindent \textbf{Claim} Let $A$ be a prime $\C$-algebra, finite over its centre $Z$ that is integrally closed. Then $Z$ is a direct summand of $A$ as a $Z$-module.\\
\textit{Proof of claim:} Since $A$ is prime, the centre $Z$ is a domain. Let $Q(Z)$ be the field of fractions of $Z$ and $D = A \otimes_Z Q(Z)$, a central simple algebra. If $\overline{Q(Z)}$ is the algebraic closure of $Q(Z)$ then 
\begin{equation}\label{eq:matrixiso}
A \otimes_Z \overline{Q(Z)} = D \otimes_{Q(Z)} \overline{Q(Z)} \simeq \textrm{Mat}_n(\overline{Q(Z)}), \textrm{ for some $n$}.
\end{equation}
Therefore we have a trace map $tr : A \otimes_Z \overline{Q(Z)} \rightarrow \overline{Q(Z)}$. It is shown in \cite[page 38]{CentralSimpleAlgebras} that one can choose the isomorphism (\ref{eq:matrixiso}) so that $tr_| : D \rightarrow Q(Z)$. Now choose $a \in A$. Since $A$ is a finite module over $Z$ there exists a monic polynomial $f \in Z[x]$ such that $f(a) = 0$. Let $g \in \overline{Q(Z)}[x]$ be the minimal polynomial of $a$, considered as an element of $\textrm{Mat}_n(\overline{Q(Z)})$ and let the roots of $g$ be $\alpha_1, \ds, \alpha_k$. Since $g \, | \, f$ in $\overline{Q(Z)}[x]$, $f(\alpha_i) = 0$ for all roots $\alpha_i$ of $g$. Therefore the algebra $B := Z[ \alpha_1, \ds, \alpha_k]$ is a finite $Z$-module. The coefficients of $g$ belong to $B$. In particular, $tr(a) \in Q(Z) \cap B = Z$ since $Z$ is assumed to be integrally closed. The restriction of $tr$ to $Z$ is just multiplication by $n$. Therefore the $Z$-module morphism $\frac{1}{n}tr$ is a left inverse to the inclusion $Z \hookrightarrow A$ and hence $Z$ is a direct summand of $A$.
\end{proof}

\begin{thm}
Fix $b$ an element of $\mathfrak{h}$ and let $\mathbf{c}'$ be the restriction of $\mathbf{c}$ to the subgroup $W_b$ of $W$. There is a Poisson isomorphism 
\bdm
\theta \, : \, \widehat{Z}_{\mathbf{c}}(W,\mathfrak{h})_b \stackrel{\sim}{\longrightarrow} \widehat{Z}_{\mathbf{c}'}(W_b,\mathfrak{h})_0.
\edm 
\end{thm}

\begin{proof}
Lemma \ref{thm:completeiso} and Proposition \ref{prop:quoiso} show that the assumptions of (\ref{prop:poissoniso}) hold. Therefore the theorem follows from Proposition \ref{prop:poissoniso}.
\end{proof} 

\begin{remark}\label{remark:either}
In Theorem \ref{thm:completeiso} it is possible to choose a point $\lambda \in \mathfrak{h}^*/W$ instead of $b \in \mf{h}/W$; the analogous statement holds.
\end{remark}

\subsection{} Let us fix $\mf{t} := (\h^{* \, W_b})^\perp \subset \h$ and $\mf{s} := \h^{W_b}$ so that $\h = \mf{t} \oplus \mf{s}$. The defining relations of $H_{t,\mathbf{c}}$ show that 
\begin{equation}\label{eq:Hiso}
H_{t,\mathbf{c}}(W_b,\mathfrak{h}) \simeq H_{t,\mathbf{c}}(W_b,\mathfrak{t}) \otimes \mathcal{D}_t(\mf{s}).
\end{equation}
Here, for a given vector space $V$, $\mathcal{D}_t(V)$ is the $\C$-algebra generated by $V$ and $V^*$: the elements of $V$ commuting amongst themselves and similarly for the elements of $V^*$, whilst $[ x, y] = t \cdot x(y)$ for $y \in V$ and $x \in V^*$. Thus, when $t \neq 0$, $\mathcal{D}_t(V)$ is isomorphic to the Weyl algebra over $V$ and when $t = 0$, $\mathcal{D}_t(V) = \C [ V \times V^*]$. However $\C [ V \times V^*]$ inherits a nondegenerate Poisson structure from $\mathcal{D}_t(V)$ given by $\{ x , x' \} = \{ y , y' \} = 0$ and $\{ x , y \} = x(y)$, for $x,x' \in V^*$ and $y,y' \in V$ (which is a particular case of the construction given in (\ref{sec:Poissondefinition})). Equivalently, $V \times V^*$ is a symplectic manifold with the canonical symplectic structure. The isomorphism (\ref{eq:Hiso}) restricts to an isomorphism of the centres. Moreover, since (\ref{eq:Hiso}) is valid for all $t$, the isomorphism of centres is a Poisson isomorphism when $t = 0$. If $\widehat{\C} [\mf{s} \times \mf{s}^*]_0$ is the completion of the polynomial ring $\C [\mf{s} \times \mf{s}^*]$ with respect to the ideal generated by $\C [\mf{s}]_+$ then there is an isomorphism of Poisson algebras 
\beq\label{eq:decompose}
Z_{\mathbf{c}'}(W_b,\mathfrak{h}) \simeq Z_{\mathbf{c}'}(W_b,\mathfrak{t}) \otimes \C [\mf{s} \times \mf{s}^*],
\eeq
which extends to an isomorphism of complete Poisson algebras
\bdm
\widehat{Z}_{\mathbf{c}'}(W_b,\mathfrak{h})_0 \simeq \widehat{Z}_{\mathbf{c}'}(W_b,\mathfrak{t})_0 \, \widehat{\otimes} \, \widehat{\C} [\mf{s} \times \mf{s}^*]_0.
\edm

\subsection{}\label{lem:parabolicdim} Fix a parabolic subgroup $W_b$ of $W$ and let $(\mathfrak{h}^{W_b})_{\textrm{reg}}$ be the set of points in $\h$ whose stablizer is $W_b$. The images of $\mathfrak{h}^{W_b}$ and $(\mathfrak{h}^{W_b})_{\textrm{reg}}$ in $\mathfrak{h}/W$ will be written $\mathfrak{h}^{(W_b)}/W$ and $\mathfrak{h}^{(W_b)}_{\textrm{reg}}/W$ respectively. They only depend on the conjugacy class of $W_b$. The sets $\mathfrak{h}^{(W_b)}_{\textrm{reg}}/W$ define a finite stratification of $\mathfrak{h}/W$ by locally closed subsets. Moreover, the closure ordering that this stratification defines agrees with the partial ordering on conjugacy classes of parabolic subgroups defined in (\ref{sec:para}) i.e.
\begin{displaymath}
(W_1) \ge (W_2) \quad \Longleftrightarrow \quad \mathfrak{h}^{(W_2)}_{\textrm{reg}} \, / \, W \subseteq \overline{\mathfrak{h}^{(W_1)}_{\textrm{reg}} \, / \, W}.
\end{displaymath}

\begin{lem}
Let $(W_b)$ be a conjugacy class of parabolic subgroups of $W$ of rank $r$, then 
\begin{displaymath}
\dim \mathfrak{h}^{(W_b)}_{\textrm{reg}} \, / \, W = n - r. 
\end{displaymath}
\end{lem}

\begin{proof}
Since $\mathfrak{h}^{(W_b)}_{\textrm{reg}}/W$ is an open subset of the irreducible variety $\mathfrak{h}^{(W_b)}/W$, $\dim \mathfrak{h}^{(W)}_{\textrm{reg}}/W = \dim \mathfrak{h}^{(W_b)}/W$. As explained in subsection (\ref{sec:para}), there is a $W'$-equivariant decomposition $\mathfrak{h} = \mathfrak{h}^{W_b} \oplus (\mathfrak{h}^{*W_b})^{\perp}$ with $\dim \, (\mathfrak{h}^{*W_b})^{\perp} = r$. Hence $\dim \mathfrak{h}^{(W_b)} = n - r$. Since the quotient map $\mathfrak{h} \twoheadrightarrow \mathfrak{h}/W$ is a finite surjective morphism, $\dim \mathfrak{h}^{(W_b)}/W = \dim \mathfrak{h}^{W_b} = n - r$.  
\end{proof}

\subsection{}\label{prop:mainidea} Recall from (\ref{sub:calogerospace}) that we have surjective morphisms $\pi_1 \, : \, X_{\mathbf{c}}(W,\h) \twoheadrightarrow \h^*/W$ and $\pi_2 \, : \, X_{\mathbf{c}}(W,\h) \twoheadrightarrow \h/W$ defined by the inclusions $\C[\mathfrak{h}]^W \hookrightarrow Z_{\mathbf{c}}(W,\h)$ and $\C[\mathfrak{h}^*]^W \hookrightarrow Z_{\mathbf{c}}(W,\h)$ respectively. The map $\Upsilon$ was defined to be $\pi_1 \times \pi_2 \, : \, X_{\mathbf{c}}(W,\h) \twoheadrightarrow \h^*/ W \times \h/W$.

\begin{prop}
Let $\mc{L}$ be a symplectic leaf in $X_{\mathbf{c}}(W,\mathfrak{h})$ of dimension $2l$. 
\begin{enumerate}
\item There exists a unique conjugacy class $(W_p)$ of parabolic subgroups of $W$ with rank $(W_p) = n - l$ such that 
\bdm
\mathcal{L} \cap \pi_1^{-1}(\mathfrak{h}_{\textrm{reg}}^{(W_p)} / W) \neq \emptyset.
\edm
\item There exists a unique conjugacy class $(W_q)$ of parabolic subgroups of $W$ with rank $(W_q) = n - l$ such that 
\bdm
\mathcal{L} \cap \pi_2^{-1}(\mathfrak{h}_{\textrm{reg}}^{*(W_q)} / W) \neq \emptyset.
\edm 
\end{enumerate}
In general $(W_p) \neq (W_q)$. 
\end{prop}

\begin{proof}
Let $P$ be the Poisson primitive ideal of $Z_{\mathbf{c}}(W,\mathfrak{h})$ defining the closure of $\mathcal{L}$ in $X_{\mathbf{c}}$. The map $\Upsilon$ is a closed, finite, surjective morphism, therefore $\Upsilon(\mathcal{L})$ is a locally closed set of dimension $2l$. It is contained in the locally closed set $\pi_1(\mathcal{L}) \times \pi_2(\mathcal{L}) \subseteq \mathfrak{h}/W \times \mathfrak{h}^*/W$. Therefore 
\begin{displaymath}
\dim \pi_1(\mathcal{L}) + \dim \pi_2(\mathcal{L}) = \dim \left( \pi_1(\mathcal{L}) \times \pi_2(\mathcal{L}) \right) \ge 2l.
\end{displaymath}
This means that either $\dim \pi_1(\mathcal{L}) \ge l$ or $\dim \pi_2(\mathcal{L}) \ge l$. For now let us assume that $\dim \pi_1(\mathcal{L}) \ge l$. Choose a conjugacy class $(W_b)$ of parabolic subgroups of minimal rank such that $\mathfrak{h}^{(W_b)}_{\textrm{reg}}/W \, \cap \, \pi_1(\mathcal{L}) \neq \emptyset$. Minimality of the rank of $(W_b)$  is equivalent to asking that the dimension of $\mathfrak{h}^{(W_b)}_{\textrm{reg}}/W$ in $\h/W$ is maximal with respect to the property $\mathfrak{h}^{(W_b)}_{\textrm{reg}}/W \, \cap \, \pi_1(\mathcal{L}) \neq \emptyset$. Since the stratification of $\mathfrak{h}/W$ by the locally closed subsets $\mathfrak{h}^{(W_b)}_{\textrm{reg}}/W$ is finite, the set $\mathfrak{h}^{(W_b)}_{\textrm{reg}}/W \, \cap \, \pi_1(\mathcal{L})$ is open in $\pi_1(\mathcal{L})$. Denote by $P'$ a prime ideal of $\widehat{Z}_{\mathbf{c}}(W,\mathfrak{h})_b$ that is minimal over the ideal $P \otimes_{Z_{\mbf{c}}(W,\mf{h})} \widehat{Z}_{\mathbf{c}}(W,\mathfrak{h})_b$. It is a Poisson primitive ideal. Let $\theta : \widehat{Z}_{\mathbf{c}}(W,\mathfrak{h})_b \stackrel{\sim}{\rightarrow} \widehat{Z}_{\mathbf{c}'}(W_b,\mathfrak{h})_0$ be the isomorphism of Theorem \ref{thm:completeiso}. Lemma \ref{lem:Poisson} says that the ideal $Q' := \theta(P') \cap Z_{\mathbf{c}'}(W_b,\mathfrak{h})$ is a Poisson primitive ideal. The isomorphism (\ref{eq:decompose}) implies that
\beq\label{eq:VQ}
\textrm{V} \, (Q') \simeq \overline{\mc{M}} \times \mf{s} \times \mf{s}^*,
\eeq
where 
\bdm
\overline{\mc{M}} = \textrm{V} \, (Q' \cap Z_{\mathbf{c}'}(W_b,\mf{t})) \subset X_{\mbf{c}'}(W_b,\mf{t}),
\edm
is the closure of some symplectic leaf $\mc{M}$. Fix rank $(W_b) = r$. Let us try to calculate the dimension of $\mc{M}$. Lemma \ref{lem:parabolicdim} says that $\dim \pi_1(\mathcal{L}) \le n - r$. Lemmata \ref{lem:primes} and \ref{lem:Poisson} show that $\height (Q') = \height (P)$. Therefore
\bdm
2l = \dim \mathcal{L} = 2n - \textrm{ht} (P) = 2n - \textrm{ht} (Q').
\edm
Since $\dim \, \mf{s} \times \mf{s}^* = 2(n - r)$, equation (\ref{eq:VQ}) shows that
\bdm
\dim \, \mc{M} + 2(n-r) = 2n - \height (Q') = 2l.
\edm
However $l \le \dim \, \pi_1(\mc{L}) \le n - r$ implies that $\dim \, \pi_1(\mc{L}) = l = n - r$ and $\dim \, \mc{M} = 0$. This also means that $\dim \, \pi_2(\mc{L}) = l$ and we could equally have choosen to work in $\h^*/W$. Clearly,
\bdm
\pi_1(\mathcal{L}) \cap \mathfrak{h}_{\textrm{reg}}^{(W_b)}/W \neq \emptyset \iff  \pi_1^{-1}( \mathfrak{h}_{\textrm{reg}}^{(W_b)}/W) \cap \mathcal{L} \neq \emptyset.
\edm
The uniqueness statement of the proposition follows from the fact that $\overline{\pi_1(\mc{L})}$ is irreducible and that $\h^{(W_b)}_{\textrm{reg}}/W \cap \pi_1(\mc{L})$ is open and dense in $\pi_1(\mc{L})$.
\end{proof}

\subsection{}\label{cor:zerodim} Let $W(\mathcal{L})$ denote the conjugacy class of parabolic subgroups of $W$ associated to $\mathcal{L}$ by Proposition \ref{prop:mainidea} (1). The partial ordering defined on the symplectic leaves of $X_{\mathbf{c}}$ by $\mathcal{L} \le \mathcal{M}$ if and only if $W(\mathcal{L}) \le W(\mathcal{M})$ in the ordering of (\ref{sec:para}) equals the partial ordering defined by the closure of leaves (c.f. \cite[(3.5)]{Poisson Orders}).

\begin{cor}
Let $\mc{L}$ be a zero dimensional symplectic leaf in $X_{\mathbf{c}}(W,\mathfrak{h})$. Then $\mc{L} \subseteq \Upsilon^{-1}(0)$.
\end{cor}

\begin{proof}
Proposition \ref{prop:mainidea} $(1)$ implies that $\mc{L} \subset \pi_1^{-1}(0)$ and Proposition \ref{prop:mainidea} $(2)$ implies that $\mc{L} \subset \pi_2^{-1}(0)$, therefore $\mc{L} \subset \pi_1^{-1}(0) \cap \pi_2^{-1}(0) = \Upsilon^{-1}(0)$. 
\end{proof}

\begin{remark}
It has been pointed out to the author by M. Martino that there is a direct proof\footnote{The idea is due to M. Martino, any errors in the argument are the authors'.} of Corollary \ref{cor:zerodim}. The rational Cherednik algebra $H_{\mbf{c}}(W,\h)$ is $\Z$-graded with $\deg \, x = 1$, $\deg \, y = -1$ and $\deg \, w = 0$ for $x \in \h \subset \C[\h^*]$, $y \in \h^* \subset \C[\h]$ and $w \in W$. The centre inherits a $\Z$-grading. Geometrically this says that there is an action of $\C^*$ on $X_{\mbf{c}}(W,\h)$. The map $\Upsilon$ is $\C^*$-equivariant and it can be shown that $0$ is the unique fixed point in $\h / W \times \h^* / W$. Since $\C^*$ is connected and the set $\Upsilon^{-1}(0)$ is finite, this is the set of $\C^*$-fixed points of $X_{\mbf{c}}(W,\h)$. It is shown in \cite[Remark 3.1]{GGOR} that there exists an element $\msf{eu} \in Z_{\mbf{c}}(W,\h)$ (the ``Euler operator''), such that $\{ \msf{eu}, z \} = (\deg z) \cdot z$ for any homogeneous element $z \in Z_{\mbf{c}}(W,\h)$ i.e. the infinitesimal action of $\C^*$ is given by the Hamiltonian vector field $\{ \msf{eu}, - \}$. Again using the fact that $\C^*$ is connected, we see that the fixed points of $X_{\mbf{c}}(W,\h)$ correspond to those closed points whose maximal ideal is preserved by $\{ \msf{eu}, - \}$. If $\mc{L}$ is zero-dimensional then the maximal ideal defining it is clearly preserved by $\{ \msf{eu}, - \}$ and therefore $\mc{L} \subset \Upsilon^{-1}(0)$.
\end{remark}

\begin{prop}\label{prop:leavesmap}
Let $(W_b)$, $b \in \h$, be a conjugacy class of parabolic subgroups of $W$ and choose a representative $W_b$ of this class. Let $\mc{T}$ denote the set of all symplectic leaves $\mc{L}$ in $X_{\mathbf{c}}(W,\mathfrak{h})$ such that $W(\mc{L}) = (W_b)$. Then there exists a surjective map
\bdm
\Psi \, : \, \{ \textrm{zero dimensional leaves of } X_{\mathbf{c}'}(W_b,\mathfrak{t}) \} \twoheadrightarrow \mc{T},
\edm
though both sets may be empty (recall that $\mf{t} = (\h^{*W_b})^\perp$).
\end{prop}

\begin{proof}
Symplectic leaves of $X_{\mbf{c}}(W,\h)$ correspond to Poisson primitive ideals of $Z_{\mbf{c}}(W,\h)$. Therefore we will define $\Psi$ in terms of Poisson primitive ideals. Since the closure $\h^{(W_b)}/W$ of $\h^{(W_b)}_{reg}/W$ in $\h/W$ is irreducible, $\h^{(W_b)}_{reg}/W$ is connected. Let $\mc{L} \in \mc{T}$. It was shown in the proof of Proposition \ref{prop:mainidea} that 
\bdm
\dim \, \h^{(W_b)}_{\textrm{reg}}/W \cap \pi_1(\mc{L}) = n - \textrm{rank} \, (W_b) = \dim \, \h^{(W_b)}_{\textrm{reg}}/W.
\edm
Therefore $\h^{(W_b)}_{\textrm{reg}}/W \cap \pi_1(\mc{L}) $ is open and dense in $\h^{(W_b)}_{\textrm{reg}}/W$. Since the number of leaves in $\mc{T}$ is finite we can choose 
\bdm
b' \in \h^{(W_b)}_{\textrm{reg}}/W \cap \bigcap_{\mc{L} \in \mc{T}} \pi_1(\mc{L}).
\edm
Without loss of generality we may assume $b' = b$. First we wish to show that there is a natural bijection between the set $\{ \textrm{zero dimensional leaves of } X_{\mathbf{c}'}(W_b,\mathfrak{t}) \} = \{ \textrm{maximal and Poisson ideals of } Z_{\mbf{c}'}(W_b,\mf{t}) \}$ and the set of Poisson primitive ideals of height $2 \, \textrm{rank} \, (W_b)$ in $\widehat{Z}_{\mbf{c}'}(W_b,\mf{h})_0$. Let $\mf{m}$ to be a maximal and Poisson ideal of $Z_{\mbf{c}'}(W_b,\mf{t})$. The isomorphism (\ref{eq:decompose}) implies that the ideal generated by $\mf{m}$ in $Z_{\mbf{c}'}(W_b,\h)$ is a Poisson primitive ideal of height $2 \, \textrm{rank} \, (W_b)$. Now set $Q = \mf{m} \otimes_{Z_{\mbf{c}}(W_b,\mf{t})} \widehat{Z}_{\mbf{c}'}(W_b,\mf{h})_0$. It follows from Lemma \ref{lem:Poisson} that $Q$ is a Poisson ideal and every prime minimal over $Q$ is Poisson primitive. Moreover, Lemma \ref{lem:primes} $(1)$ says that the height of each of these minimal primes is $2 \, \textrm{rank} \, (W_b)$. Therefore it suffices to show that $Q$ is itself prime. Noting that 
\bdm
Q = \left( \mf{m} \otimes_{Z_{\mbf{c}}(W_b,\mf{t})} \widehat{Z}_{\mbf{c}'}(W_b,\mf{t})_0 \right) \otimes_{\widehat{Z}_{\mbf{c}'}(W_b,\mf{t})_0} \widehat{Z}_{\mbf{c}'}(W_b,\mf{h})_0,
\edm
repeated applications of Lemma \ref{lem:polyprimes} reduces the question to showing that $\mf{m} \otimes_{Z_{\mbf{c}}(W_b,\mf{t})} \widehat{Z}_{\mbf{c}'}(W_b,\mf{t})_0$ is prime. But this follows from Lemma \ref{lem:primes} $(3)$, since Corollary \ref{cor:zerodim} shows that the ideal generated in $Z_{\mbf{c}}(W_b,\mf{t})$ by the space $\mf{n}(0)$ is contained in $\mf{m}$. The definition of $\Psi$ is now straight-forward: by Theorem \ref{thm:completeiso} we may consider $Z_{\mbf{c}}(W,\h)$ to be a subalgebra of $\widehat{Z}_{\mbf{c}'}(W_b,\mf{h})_0$ then 
\bdm
\Psi(\mf{m}) := Z_{\mbf{c}}(W,\mf{h}) \cap Q.
\edm
Lemmata \ref{lem:primes} and \ref{lem:Poisson} show that $\Psi(\mf{m})$ is a Poisson primitive ideal of height $2r$. The surjectivity of $\Psi$ follows from the fact that each prime minimal over $P \otimes_{Z_{\mbf{c}}(W,\mf{h})} \widehat{Z}_{\mathbf{c}}(W,\mathfrak{h})_b$, $P \in \mc{T}$, corresponds to some zero dimensional leaf in $X_{\mbf{c}'}(W_b,\mf{t})$. 
\end{proof}

\begin{remark}
It is natural to ask
\begin{center}
\textit{Q. Is the map $\Psi$ a bijection?}
\end{center}
It can be seen from the proof of Proposition \ref{prop:leavesmap} that $| \Psi^{-1}(\mc{L})|$ equals the number of minimal primes over $P \otimes_{Z_{\mbf{c}}(W,\mf{h})} \widehat{Z}_{\mathbf{c}}(W,\mathfrak{h})_b$ (where $P$ is the Poisson primitive ideal defining the closure of $\mc{L}$). Therefore the above question is equivalent to showing that $P \otimes_{Z_{\mbf{c}}(W,\mf{h})} \widehat{Z}_{\mathbf{c}}(W,\mathfrak{h})_b$ is prime in $\widehat{Z}_{\mathbf{c}}(W,\mathfrak{h})_b$.
\end{remark}

\subsection{}\label{cor:underformedleaves} If $\mbf{c} = 0$ then we recover a result by Brown and Gordon \cite[Proposition 7.7]{Poisson Orders}, removing the requirement that $W$ be a Weyl group.

\begin{cor}
Let $W$ be a complex reflection group, $\h$ its reflection representation. Then the number of symplectic leaves of dimension $2l$ in $\h \times \h^* \, / \, W$ equals the number of conjugacy classes of parabolic subgroups of $W$ of rank $\dim \, \h - l$.
\end{cor}

\begin{proof}
Let $W_b$, $b \in \h$ be a parabolic subgroup of $W$ of rank $r$, $\mf{t} \subset \h$ its reflection representation. Then $\{ 0 \}$ is the unique zero dimensional symplectic leaf in $\mf{t} \times \mf{t}^* \, / \, W_b$. Therefore Proposition \ref{prop:leavesmap} implies that there exists a unique symplectic leaf in $\h \times \h^* \, / \, W$ labelled by $(W_b)$ and this leaf has dimension $2 \dim \h - 2r$.
\end{proof}

\section{Cuspidal representations for $H_{0,\mbf{c}}(W)$}

\subsection{} A closed point $\chi \in X_{\mathbf{c}}(W,\mathfrak{h})$ can be regarded as a non-zero algebra homomorphism $\chi : Z_{\mathbf{c}}(W,\mathfrak{h}) \rightarrow \C$. We define 
\begin{displaymath}
H_{\mathbf{c},\chi} := \frac{H_{0,\mathbf{c}}(W,\mathfrak{h})}{\langle \ker\, \chi \rangle},
\end{displaymath}
a finite dimensional quotient of $H_{0,\mathbf{c}}(W,\mathfrak{h})$. 

\begin{defn}
The algebra $H_{\mathbf{c},\chi}$ is said to be a \textit{cuspidal algebra} if  $\{ \chi \}$ is a zero dimensional leaf of $X_{\mathbf{c}}$. A simple $H_{\mathbf{c}}(W,\mathfrak{h})$-module $L$ is a \textit{cuspidal representation} if $L$ is a module for some cuspidal algebra $H_{\mathbf{c},\chi}$, or equivalently, Supp $L$ is a zero dimensional symplectic leaf in $X_{\mathbf{c}}$.
\end{defn}

Note that the space $X_{\mathbf{c}}(W,\mathfrak{h})$ may have no zero dimensional leaves. For instance, if $W = S_n, n > 1$ and $\mbf{c} \neq 0$ then it is shown in \cite[Corollary 1.14]{1} that $X_{\mbf{c}}$ is a symplectic manifold of dimension $2n$ and has no zero dimensional leaves.

\subsection{Flows along symplectic leaves}
The algebra $H_{\mbf{c}}(\h,W)$ can be considered as a sheaf of algebras on $X_{\mbf{c}}(W,\h)$. The fibre of this sheaf at a point $\chi \in X_{\mbf{c}}(W,\h)$ is $H_{\mathbf{c},\chi}$. Let $\mathcal{L}$ be a leaf in $X_{\mathbf{c}}$ and $\chi_1, \chi_2 \in \mathcal{L}$. Then we have the beautiful result \cite[Theorem 4.2]{Poisson Orders}, based on \cite[Corollary 9.2]{DeConciniLyubashenko}:
\begin{equation}\label{eq:BGiso}
\psi_{\chi_,\chi_2} : H_{\mathbf{c},\chi_1} \stackrel{\sim}{\longrightarrow} H_{\mathbf{c},\chi_2}
\end{equation}
i.e. the representation theory of $H_{\mathbf{c}}(W,\mathfrak{h})$ is constant along the leaves of $X_{\mathbf{c}}(W,\mathfrak{h})$. We wish to show that this isomorphism is $W$-equivariant. 

\subsection{} We recall here the construction of the isomorphism (\ref{eq:BGiso}) as given in \cite[Theorem 4.2]{Poisson Orders}. Fix $H = H_{\mathbf{c}}(W,\mathfrak{h})$, $Z = Z_{\mbf{c}}(W,\h)$ and let $P$ be the Poisson prime defining the closure of $\mc{L}$. Then $H / P \cdot H$ is a $Z / P$-module and the algebras $H_{\mathbf{c},\chi_1}$ and $H_{\mathbf{c},\chi_2}$ are quotients of $H / P \cdot H$. The construction of (\ref{sec:Poissondefinition}) defines an action of $f \in Z$ on $H$ as a derivation, $D_f(a) := \{ f , a \}$ for $a \in H$. This makes $H$ into a Poisson module for $Z$. By \cite[Lemma 4.1]{Poisson Orders}, $H / P \cdot H$ is a $Z / P$-Poisson module with action induced from the derivations $D_f, f \in Z$. It is shown in the proof of \cite[Theorem 4.2]{Poisson Orders} that $H / P \cdot H$ is a locally free sheaf when restricted to $\mc{L}$. The space $\mc{L}$ is a smooth quasi-projective variety and we will now consider it as a complex analytic variety. Let $\hat{Z}$ be the algebra of holomorphic functions on $\mc{L}$ and define $\hat{H} = H \otimes_{(Z/P)} \hat{Z}$. The derivations $D_f$ extend to derivations on $\hat{H}$ because the Poisson structure extends uniquely to $\hat{Z}$. For each point $\chi \in \mc{L}$, the natural map $H_{\mbf{c},\chi} \rightarrow \hat{H}_{\chi}$ is an algebra isomorphism. Any two points $\chi_1$ and $\chi_2$ on $\mc{L}$ can be connected by a finite number of Hamiltonian flows, it is these flows that induce the isomorphism (\ref{eq:BGiso}).\\

Therefore we may assume that there exists $f \in \hat{Z}$ and a Hamiltonian flow $\rho : B \rightarrow \mc{L}$ for $f$ (where $B \subset \C$ is a small disk around zero) such that $\rho(0) = \chi_1$ and $\rho(t) = \chi_2$. Shrinking $B$ if necessary and choosing an open neighbourhood $U$ of $\rho(B)$ in $\mc{L}$, we may assume by Darboux's Theorem that we are in the following explicit situation: $U \subset \C^{2m}$ is an open, simply connected set containing $\chi_1, \chi_2$; $\mc{O}_U$ is the sheaf of holomorphic functions on $U$ and $x_1, \ds, x_{m}, y_1, \ds, y_m$ are symplectic coordinates on $U$. That is, there is a non-degenerate Poisson bracket on $\mc{O}_U$ defined by $\{ x_i , y_{j} \} = \delta_{ij}$ and $\{ x_i ,x_j \} = \{y_i,y_j \} =  0$ for all $1 \le i,j \le m$. Then $H' := \hat{H} \otimes_{\hat{Z}} Z'$ is an algebra containing $Z' = \mc{O}_U(U)$ such that $H' = \bigoplus_{i= 1}^n Z' \cdot a_i$ is free as a $Z'$-module. The action of $D_f$ on $H'$ is defined by
\bdm
D_f(x_i) = c_i(x,y), \quad D_f(y_j) = d_j(x,y) \quad \textrm{ and } \quad D_f(a_i) = \sum_{j = 1}^n e_{ij}(x,y) a_j,
\edm
for some functions $c_i,d_i,e_{ij} \in \mc{O}_U$. The algebra $H'$ is the space of global sections of the trivial vector bundle $U \times \C^n$ over $U$. We fix coordinates $z_1, \ds, z_n$ on $\C^n$ such that $z_i(a_j) = \delta_{ij}$. Then the derivative $D_f$ can be expressed explicitly as
\bdm
D_f = \sum_{i = 1}^{m} \left( c_i(x,y) \frac{\p}{\p x_i} + d_i(x,y) \frac{\p}{\p y_i} \right) - \sum_{i,j = 1}^n e_{ji}(x,y) z_j \frac{\p}{\p z_i},
\edm
the minus sign appears because the $z_i$ are dual to the $a_i$. The flow $\rho = (\rho_1, \ds, \rho_{m}, \rho_1', \ds, \rho_m')$ on $U$ with respect to $D_f$ satisfies $D_f(h)(\rho(t)) = \frac{d \rho}{d t}(h)(t)$ for all $h \in \mc{O}_U$ and is given explicitly as the solution to the system of equations
\beq\label{eq:floweq}
\frac{d \rho_i}{d t} = c_i(\rho(t)), \quad \frac{d \rho_i'}{d t} = d_i(\rho(t)), \quad 1 \le i \le m.
\eeq
It is clear from the presentation that $D_f$ actually defines a derivation of $\mc{O}_U[z_1, \ds, z_n]$. Every flow $\Psi : B \rightarrow U \times \C^n$ for $D_f$ is a lift of a flow $\rho : B \rightarrow U$. This means that there exists some function $\psi : B \rightarrow \C^n$ such that $\Psi = (\rho,\psi)$. Explicitly, $\psi$ satisfies the system of equations
\beq\label{eq:liftedeq}
\frac{d \psi_i}{d t} = - \sum_{j = 1}^n e_{ji}(\rho(t)) \psi_j(t) \quad 1 \le i \le n.
\eeq
Since this is a linear system of equations, the induced map on fibres $\psi_{\chi_,\chi_2} : \hat{H}_{\chi_1} \rightarrow \hat{H}_{\chi_2}$ is linear. It is proven in \cite[Theorem 4.2]{Poisson Orders} that $\psi_{\chi_1,\chi_2}$ is actually an algebra isomorphism. 

\subsection{}\label{lem:equivariantflow}
Any section $w \in H'$ can be considered as a function $w \circ \rho : B \rightarrow U \times \C^n$ extending the flow $\rho$. Locally, there is a unique flow $\Psi \, : \, B \rightarrow U \times \C^n$ for $D_f$, lifting $\rho$ and satisfying $\Psi(0) = w \circ \rho(0)$. 

\begin{lem}
If $w \in H'$ such that $D_f(w) = 0$ then $\Psi = w \circ \rho$. 
\end{lem}

\begin{proof}
By the uniqueness of flows it suffices to show that $w \circ \rho$ is a flow. Let us write $w = \sum_{i = 1}^n g_i(x,y) a_i$. Then $D_f(w) = 0$ implies that 
\bdm
\sum_{i,j = 1}^n \left( c_j(x,y) \frac{\p g_i}{\p x_j} + d_j(x,y) \frac{\p g_i}{\p y_j} \right) a_i + \sum_{i,j = 1}^n g_i e_{ij}(x,y) a_j = 0,
\edm
hence
\beq\label{eq:identiyone}
\sum_{j = 1}^n \left( c_j(x,y) \frac{\p g_i}{\p x_j} + d_j(x,y) \frac{\p g_i}{\p y_j} \right) + \sum_{j = 1}^n g_j e_{ji}(x,y) = 0, \quad \forall \, 1 \le i \le n.
\eeq
Equation (\ref{eq:liftedeq}) shows that it suffices to prove that $\frac{d (g_i \circ \rho)}{d t} = - \sum_{j = 1}^n e_{ji}(\rho(t)) g_j \circ \rho(t)$. Using the chain rule, (\ref{eq:floweq}) and (\ref{eq:identiyone}),
\bdm
\frac{d (g_i \circ \rho)}{d t} = \sum_{j = 1}^n \frac{\p g_i}{\p x_j}(\rho(t)) \cdot \frac{d \rho_j}{d t} + \sum_{j = 1}^n \frac{\p g_i}{\p y_j}(\rho(t)) \cdot \frac{d \rho_j'}{d t}
\edm
\bdm
= \sum_{j = 1}^n \left( \frac{\p g_i}{\p x_j}(\rho(t)) \cdot c_j(\rho (t)) + \frac{\p g_i}{\p y_j}(\rho(t)) \cdot d_j(\rho (t)) \right) = - \sum_{j = 1}^n g_j(\rho(t)) e_{ji}(\rho(t)).
\edm
\end{proof}

\begin{cor}
Let $\chi_, \chi_2$ be points on the leaf $\mc{L}$. Then the algebra isomorphism $\psi_{\chi_,\chi_2} : H_{\mathbf{c},\chi_1} \stackrel{\sim}{\rightarrow} H_{\mathbf{c},\chi_2}$ is $W$-equivariant. 
\end{cor}

\begin{proof}
As explained above, the isomorphism $\psi_{\chi_,\chi_2}$ is the composition of finitely many isomorphisms induced from local Hamiltonian flows on $\mc{L}$. Therefore we may assume that we are in the explicit local situation described above. Let $w \in W$ and $a \in \hat{H}_{\chi_1}$. We wish to show that $\psi_{\chi_1,\chi_2}(w \cdot a) = w \cdot \psi_{\chi_1,\chi_2}(a)$. Since $\psi_{\chi_1,\chi_2}$ is an algebra morphism this is equivalent to proving that $\psi_{\chi_1,\chi_2}(\bar{w}) = \bar{w}$ where $\bar{w}$ is the image of $w$ in $\hat{H}_{\chi_1}$ and $\hat{H}_{\chi_2}$ respectively. From the construction of the derivations $D_f$ as given in (\ref{sec:Poissondefinition}) we see that $D_f(w) = 0$ for all $f \in \hat{Z}$. In terms of the trivialization of $\hat{H}$ over $U$, $\bar{w} = w \circ \rho(0) \in \hat{H}_{\chi_1}$ and $\bar{w} = w \circ \rho(t) \in \hat{H}_{\chi_2}$ (where $t \in B$ such that $\rho(t) = \chi_2$). Thus the result is a consequence of Lemma \ref{lem:equivariantflow}.
\end{proof}

\subsection{}\label{thm:main} We can now state the main result of this section.

\begin{thm}
Let $\mathcal{L}$ be a leaf in $X_{\mathbf{c}}(W,\mathfrak{h})$ of dimension $2l$ and $\chi$ a point on $\mathcal{L}$. Then there exists a parabolic subgroup $W_b$, $b \in \h$, of $W$ of rank $\dim \mathfrak{h} - l$ and a cuspidal algebra $H_{\mathbf{c}',\psi}$ with $\psi \in X_{\mathbf{c}'}(W_b,\mf{t})$ (recall that $\mf{t} = (\mathfrak{h}^{W_b})^\perp$) such that
\begin{displaymath}
H_{\mathbf{c},\chi} \simeq \textrm{Mat}_{\, |W / W_b|}\,(H_{\mathbf{c}',\psi}).
\end{displaymath}
\end{thm}

\begin{proof}
By Proposition \ref{prop:mainidea} there exists a unique conjugacy class $(W_b)$ of parabolic subgroups of $W$ such that $\mc{L} \, \cap \, \pi_1^{-1}( \h_{\textrm{reg}}^{(W_b)}/W) \neq \emptyset$. Without loss of generality, $b \in \pi_1(\mc{L}) \cap \h_{\textrm{reg}}^{(W_b)}/W$. Using the isomorphism (\ref{eq:BGiso}) we may assume that $\chi \in \mc{L} \cap \pi_1^{-1}(b)$. Let $K = \Ker \chi$. Then $K \otimes_{Z_{\mathbf{c}}(W,\mathfrak{h})} \widehat{Z}_{\mathbf{c}}(W,\mathfrak{h})_b$ is a maximal ideal in $\widehat{Z}_{\mathbf{c}}(W,\mathfrak{h})_b \simeq \widehat{Z}_{\mathbf{c}'}(W_b,\mathfrak{h})_0$ and the arguments in the proof of Propsition \ref{prop:mainidea} show that $M = Z_{\mbf{c}'}(W_b,\mf{t}) \cap K$ is a maximal and Poisson ideal of $Z_{\mbf{c}'}(W_b,\mf{t})$. If $N = Z_{\mbf{c}'}(W_b,\mf{h}) \cap K$ then the isomorphism (\ref{eq:Hiso}) shows that
\bdm
H_{0,\mathbf{c}'}(W_b,\mathfrak{h}) / N \cdot H_{0,\mathbf{c}'}(W_b,\mathfrak{h}) \simeq H_{0,\mathbf{c}'}(W_b,\mathfrak{t}) / M \cdot H_{0,\mathbf{c}'}(W_b,\mathfrak{t})
\edm
is some cuspidal quotient $H_{\mathbf{c}',\psi}$ of $H_{0,\mathbf{c}'}(W_b,\mathfrak{t})$ (here $\Ker \, \psi = M$). Now the isomorphism of Theorem \ref{thm:BEiso} induces an isomorphism
\bdm
\theta \, : \, H_{\mathbf{c},\chi} = \widehat{H}_{0,\mathbf{c}} (W,\mathfrak{h})_b / K \cdot \widehat{H}_{0,\mathbf{c}} (W,\mathfrak{h})_b \rightarrow C(W,W_b,\widehat{H}_{0,\mathbf{c}'}(W_b,\mathfrak{h})_0 / N \cdot \widehat{H}_{0,\mathbf{c}'}(W_b,\mathfrak{h})_0) \simeq \textrm{Mat}_{\, |W / W_b|}\,(H_{\mathbf{c}',\psi}).
\edm
\end{proof}

\begin{remark} There is a canonical finite dimensional quotient of the rational Cherednik algebra, the \textit{restricted rational Cherednik algebra}. We refere the reader to \cite{Baby} for the definition. Let $H_{\mbf{c},\chi}$ be a cuspidal algebra. Corollary \ref{cor:zerodim} shows that there is a block $B$ of the restricted rational Cherednik algebra $\bar{H}_{\mathbf{c}}(W)$ such that 
\begin{displaymath}
H_{\mathbf{c},\chi} = \frac{B}{Z_{\mbf{c}}(W) \cap B}.
\end{displaymath}
In particular, every cuspidal module occurs as a simple module for the restricted rational Cherednik algebra.
\end{remark}

\begin{prop}
Choose a point $\chi \in \mc{L}$ and let $(W_b)$ be the conjugacy class of parabolic subgroups labelling $\mc{L}$ (as in Proposition \ref{prop:mainidea} (1)). Then there exists a cuspidal algebra $H_{\mathbf{c}',\psi}$ for $W_b$ and functor 
\bdm
\Phi_{\psi, \chi} \, : \, H_{\mathbf{c}',\psi} \mmod \stackrel{\sim}{\longrightarrow} H_{\mbf{c},\chi} \mmod
\edm
defining an equivalence of categories such that
\begin{displaymath}
\Phi_{\psi,\chi}(M) \simeq \textsf{Ind}_{\, W_b}^{\, W} \,  M \quad \forall \, M \in H_{\mathbf{c}',\psi} \mmod
\end{displaymath}
as $W$-modules.
\end{prop}

\begin{proof}
If $M$ is any $\widehat{H}_{0,\mathbf{c}'} (W_b,\mathfrak{t})_0$-module and $\theta$ the isomorphism of Theorem \ref{thm:BEiso}, then $\theta^* (M) = \textrm{Fun}_{W_b}(W,M)$. As a $W$-module, $\textrm{Fun}_{W_b}(W,M) \simeq \textsf{Ind}_{\, W_b}^{\, W} \,  M$. Taking $\chi' \in \pi^{-1}(b) \cap \mc{L}$ and fixing an isomorphism $\phi_{\chi',\chi} : H_{\mbf{c},\chi'} \stackrel{\sim}{\rightarrow} H_{\mbf{c},\chi'}$ as in (\ref{eq:BGiso}) defines an equivalence $(\phi_{\chi',\chi})_* : H_{\mathbf{c},\chi'} \mmod \stackrel{\sim}{\rightarrow} H_{\mbf{c},\chi} \mmod$. Corollary \ref{lem:equivariantflow} says that $\phi_{\chi',\chi}$ is $W$-equivariant therefore $\Phi_{\psi, \chi} = (\phi_{\chi',\chi})_* \circ \theta^*$ has the desired property.
\end{proof}

\begin{example}
Let $I_2(m) = \langle a,b, \, | \, a^m = b^2 = 1, bab = a^{-1} \rangle$ be the dihedral group of order $2m$. When $m$ is odd there is only one conjugacy class of reflections, $\{ a^s b \, | \, 0 \le s \le m-1 \}$, and when $m$ is even there are two, $C_1 = \{ a^s b \, | \, 0 \le s \le m-1, s \textrm{ even} \}$ and $C_2 = \{ a^s b \, | \, 0 \le s \le m-1, s \textrm{ odd} \}$. The dihedral groups are rank two reflection groups therefore $\dim X_{\mbf{c}}(I_2(m)) = 4$ and, for $m \ge 5$, it is always a singular variety as shown in \cite{Baby}. The conjugacy classes of parabolic subgroups in $I_2(m)$ are $(1)$, $(\langle b \rangle)$ and $(I_2(m))$ when $m$ is odd and $(1)$, $(\langle b \rangle)$, $(\langle ab \rangle)$ and $(I_2(m))$ when $m$ is even. By making use of Corollary \ref{cor:zerodim} and knowing the blocks of the restricted rational Cherednik algebra, which the author has calculated in his PhD thesis, one can show that the symplectic leaves for $X_{\mbf{c}}(I_2(m))$ are described as follows. 

\begin{table}[h]\label{tab:tab1}
\centering
\caption{Label, dimension and number of leaves for $I_2(m)$, $m$ even}
\begin{tabular}{c|c|cccc}
      &    & \# of leaves  &                                &                          &                 \\
label & dim & $\mbf{c} = 0$& $\mbf{c} \in \{ 0 \} \times \C^{\times}$ & $\mbf{c} \in \C^{\times} \times \{ 0 \}$ & $\mbf{c}$ generic \\
\hline\hline
$(1)$  & 4   & 1           & 1                           & 1                            & 1                \\ 
$(\langle b \rangle)$ & 2  & 1          & 1              & 0                            & 0                \\
$(\langle ab \rangle)$& 2  & 1          & 0              & 1                            & 0                \\
$(I_2(m))$& 0 & 1          & 1                           & 1                            & 1              
\end{tabular}
\end{table}
\begin{table}
\caption{Label, dimension and number of leaves for $I_2(m)$, $m$ odd}
\begin{tabular}{c|c|cccc}
      &    & \# of leaves  &                                &                          &                 \\
label & dim & $\mbf{c} = 0$& $\mbf{c} \neq 0$ \\
\hline\hline
$(1)$  & 4   & 1           & 1    \\ 
$(\langle b \rangle)$ & 1  & 1 & 0    \\
$(I_2(m))$& 0 & 1          & 1                 
\end{tabular}
\end{table}
In all cases, if $\chi$ is a point on a two dimensional leaf then $H_{\mbf{c},\chi}$ is isomorphic to six by six matrices over the cuspidal algebra $\C[ x,y] \rtimes S_2 / (x^2, xy, y^2)$. When $m = 6$, $I_2(6)$ is the Weyl group $G_2$. In this case, the cuspidal algebra supported on the zero dimensional leaf is a quotient of the algebra described in \cite[Remark 16.5 (i)]{1}.
\end{example}  

\section*{Acknowledgements}

The research described here was done both at the University of Edinburgh with the financial support of the EPSRC and during a visit to the University  of Bonn with the support of a DAAD scholarship. This material will form part of the author's PhD thesis for the University of Edinburgh. The author would like to express his gratitude to his supervisor, Professor Iain Gordon, for his help, encouragement and patience. He also thanks Dr. Maurizio Martino, Dr. Maria Chlouveraki and Professor Ken Brown for many fruitful discussions.

\bibliographystyle{plain}
\bibliography{biblo}

\end{document}